\documentclass[12pt,twoside]{article}

\usepackage[a4paper]{geometry}
\makeatletter
\renewcommand\title[1]{\gdef\@title{\reset@font\Large\bfseries #1}}
\renewcommand\section{\@startsection {section}{1}{\z@}%
                                   {-3.5ex \@plus -1ex \@minus -.2ex}%
                                   {2.3ex \@plus.2ex}%
                                   {\normalfont\large\bfseries}}
\renewcommand\subsection{\@startsection{subsection}{2}{\z@}%
                                     {-3ex\@plus -1ex \@minus -.2ex}%
                                     {1.5ex \@plus .2ex}%
                                     {\normalfont\normalsize\bfseries}}
\renewcommand\subsubsection{\@startsection{subsubsection}{3}{\z@}%
                                     {-2.5ex\@plus -1ex \@minus -.2ex}%
                                     {1.5ex \@plus .2ex}%
                                     {\normalfont\normalsize\bfseries}}

\def\@runningauthor{}\newcommand{\runningauthor}[1]{\def\runningauthor{#1}}
\def\@runningtitle{}\newcommand{\runningtitle}[1]{\def\runningtitle{#1}}

\renewcommand{\ps@plain}{%
\renewcommand{\@evenhead}{\footnotesize\scshape \hfill\runningauthor\hfill}
\renewcommand{\@oddhead}{\footnotesize\scshape \hfill\runningtitle\hfill}}
\g@addto@macro\bfseries{\boldmath}

\makeatother

\usepackage{amsthm,amsmath,amssymb}

\usepackage{graphicx}

\usepackage[colorlinks=true,citecolor=black,linkcolor=black,urlcolor=blue]{hyperref}

\theoremstyle{plain}
\newtheorem{theorem}{Theorem}
\newtheorem{lemma}[theorem]{Lemma}
\newtheorem{corollary}[theorem]{Corollary}
\newtheorem{proposition}[theorem]{Proposition}

\theoremstyle{definition}

\theoremstyle{remark}
\newtheorem{remark}[theorem]{Remark}

\def \F {{\mathbb F}}

\def \Tr {{\rm Tr_n}}
\def \T {{\rm Tr}}

\def \V {\mathbb{V}}

\title{On the normality of $p$-ary bent functions}

\runningtitle{Normality for $p$-ary bent functions}

\author{Wilfried Meidl\thanks{First Author is supported by the Austrian Science Fund (FWF) Project no.\ M1767-N26.}\\
\small Johann Radon Institute for Computational and Applied Mathematics \\[-0.8ex]
\small Austrian Academy of Sciences, Linz, Austria, and \\ 
\small Otto-von-Guericke University Magdeburg, \\ [-0.8ex]
\small Universit\"atsplatz 2, 39106 Magdeburg, Germany \\[-0.8ex]
\small\tt meidlwilfried@gmail.com\\
\and
\'Isabel Pirsic \thanks{Second Author is supported by the Austrian Science Fund (FWF) Project no.\ P27351-N16}\\
\small Institute of Financial Mathematics and Applied Number Theory\\[-0.8ex]
\small Johannes Kepler Universit\"at\\[-0.8ex]
\small Linz, Austria \\
\small\tt isa.pirsic@gmail.com
}

\runningauthor{W.\ Meidl, \'I.\ Pirsic}

\date{}

\begin{document}

\maketitle

\thispagestyle{empty}

\begin{abstract}
Depending on the parity of $n$ and the regularity of a bent function $f$ from $\F_p^n$ to $\F_p$,
$f$ can be affine on a subspace of dimension at most $n/2$, $(n-1)/2$ or $n/2-1$. We point out that
many $p$-ary bent functions take on this bound, and it seems not easy to find examples for which one 
can show a different behaviour. This resembles the situation for Boolean bent functions of which many 
are (weakly) $n/2$-normal, i.e. affine on a $n/2$-dimensional subspace. However applying an algorithm 
by Canteaut et.al., some Boolean bent functions were shown to be not $n/2$-normal. We develop an algorithm 
for testing normality for functions from $\F_p^n$ to $\F_p$. Applying the algorithm, for some bent functions 
in small dimension we show that they do not take on the bound on normality. Applying direct sum of functions
this yields bent functions with this property in infinitely many dimensions.
\end{abstract}

%%%%%%%%%%%%%%%%%%%%%%%%%%%%%%%%%%%%%%%%%%%%%%%%%%%%%%%

\section{Introduction}

Let $p$ be a prime, and let $f$ be a function from an $n$-dimensional vector space $\V_n$ over $\F_p$ to $\F_p$.
The {\it Walsh transform} of $f$ is the complex valued function
\[ \widehat{f}(u) = \sum_{x\in \V_n}\epsilon_p^{f(x)-\langle u,x\rangle}, \quad \epsilon_p = e^{2\pi i/p},  \]
where $\langle u,x\rangle$ is a (nondegenerate) inner product in $\V_n$. The classical frameworks are $\V_n = \F_p^n$,
in which case we take the conventional dot product as inner product, and $\V_n = \F_{p^n}$, for which the standard inner 
product is $\langle u,x\rangle = \Tr(ux)$, where $\Tr(z)$ is the absolute trace of $z$ in $\F_{p^n}$.

The function $f$ is called a {\it bent function} if $|\widehat{f}(b)| = p^{n/2}$ for all $b\in \V_n$.
For Boolean bent functions we have $\widehat{f}(b) = (-1)^{f^*(b)}2^{n/2}$ for a Boolean function $f^*$, 
called the dual of $f$. When $p$ is odd, then a bent function $f$ satisfies (cf. \cite{hk})
\begin{equation}
\label{(2)} 
\widehat{f}(b) =
\left\{\begin{array}{r@{\quad:\quad}l}
\pm \epsilon_p^{f^*(b)}p^{n/2} & p^n \equiv 1\bmod 4; \\
\pm i\epsilon_p^{f^*(b)}p^{n/2} & p^n \equiv 3\bmod 4,
\end{array}\right.
\end{equation}
for a function $f^*$ from $\V_n$ to $\F_p$. Accordingly $f$ is called 
{\it regular} if $p^{-n/2}\widehat{f}(b) =\epsilon_p^{f^*(b)}$ for all $b \in \V_n$,
which for a Boolean bent function always holds.
If $p^{-n/2}\widehat{f}(b) =\zeta\ \epsilon_p^{f^*(b)}$ for some $\zeta\in\{\pm 1,\pm i\}$, independent from $b$,
we call $f$ {\it weakly regular}, otherwise $f$ is called {\it non-weakly regular}.
Note that regular implies weakly regular. Weakly regular bent functions always come in pairs, since the dual is bent as well.
This does in general not hold for non-weakly regular bent functions, see \cite{cmp0,cmp1}.
Note that Boolean bent functions only exist for even $n$, which is different when $p$ is odd, where bent functions exist in
even and in odd dimension.

Bent functions are interesting objects due to applications in cryptography and coding, and due to rich connections to objects
in combinatorics and geometry. In particular they define relative difference sets in the elementary abelian $p$-group.
Many constructions and infinite classes of bent functions are known, hence research focuses on the nature and properties of bent functions,
rather than on discovering new formulas for bent functions. In this article we investigate normality for $p$-ary bent functions,
which then also describes a feature of the corresponding relative difference set. 

A function $f:\V_n\rightarrow \F_p$ is called {\it $k$-normal} if there exists a $k$-dimensional affine subspace of $\V_n$ restricted to which
$f$ is constant. If $f$ is affine on a $k$-dimensional affine subspace of $\V_n$, then $f$ is called {\it weakly-$k$-normal}.
When $n$ is even and $k=n/2$, then $f$ is called {\it (weakly)-normal}.

Many classical constructions of Boolean bent functions like Maiorana-McFarland and $PS^+$ yield normal functions. This is very different for 
random Boolean functions, which are not likely to be constant (affine) on an affine subspace with ``large'' dimension \cite{c04}.
The question if there exist non-(weakly)-normal Boolean bent functions was open for several years. In \cite{cddl} it was shown that the Kasami 
bent function in dimension $14$ is non-weakly-normal. Non-weakly-normal bent functions in dimension $10$ (and $12$) were presented in \cite{lg}.
By \cite[Lemma 25]{cddl} this guarantees the existence of non-weakly-normal Boolean bent functions in (even) dimension $n \ge 10$.

$k$-normality may also be of cryptographic significance. As pointed out in \cite{c04}, $k$-normality is a quite natural complexity criterion,
since any affine function is constant on an affine hyperplane. Moreover there is a relation between normality and nonlinearity for Boolean functions,
see \cite[Proposition 2]{c04}. As also mentioned in \cite{c04}, the $k$-normality was not yet related to explicit attacks on ciphers, however the 
situation was the same for nonlinearity when it was introduced. In fact, meanwhile the attack on the stream cipher Grain-128 in \cite{mgpi} is based on the
$5$-normality of the $9$-variable filter function (which can be seen as a modification of the standard quadratic bent function in $8$ variables),
used in the sequence generation.

The situation for bent functions from $\V_n$ to $\F_p$, $p$ odd,
is somewhat different from the Boolean case. 
In \cite{cmp} it is pointed out that a weakly regular but not regular bent function in even dimension $n$ cannot be normal. However some results 
indicate that also for odd $p$, bent functions exhibit a typical normality behaviour. It may not be easy find bent functions for which one can 
prove a different behaviour.

In this paper, we first present a $p$-ary equivalent of a result of Carlet in \cite{c04} showing that - as one would expect - an arbitrary 
$p$-ary function is with high probability not (weakly)-$k$-normal for any not very small value of $k$. We then show the $p$-ary equivalent
of a relation between nonlinearity and normality for Boolean functions, \cite{c04,c}.
We summarize some known results on normality for $p$-ary bent functions, which indicate that many have a "typical"
behaviour with respect to normality, similar as it was observed in the Boolean case: Many $p$-ary bent functions are 
$k$-normal, where $k$ is as large as it is theoretically possible for a bent function.

In Section \ref{test} we present an algorithm for testing (weak)-$k$-normality for $p$-ary functions. Our algorithm is not
a straightforward generalization of the algorithm in \cite{cddl}, which was used to find non-weakly-normal Boolean bent 
functions in dimension $14$ \cite{cddl}, and $10$ and $12$ \cite{lg}. Applying this algorithm we find the first examples
of $p$-ary bent functions (in small dimensions) which do not possess $k$-normality with maximal possible $k$. Generalizing 
Lemma 25 of \cite{cddl} we then can obtain bent functions with this property in every larger dimension of the same parity.

\section{Normality results}

One target in this paper is to pave the way for a systematic analysis of the behaviour of $p$-ary bent functions with respect to normality.
We hence start with showing some $p$-ary equivalents of results on the normality behaviour of Boolean (bent) functions. 
Our first proposition, is the $p$-ary version of Theorem 3 and Proposition 1 in \cite{c04}.
The proof resembles the proof in \cite{c04}.
\begin{proposition}
\label{hnn}
Let $k_n$ be a sequence of integers %(greater than $2$?) 
such that $\lim_{n\rightarrow\infty}\frac{p^{k_n}}{nk_n} = \infty$.
The density in the set of functions from $\V_n$ to $\F_p$ of the functions which are weakly-$k_n$-normal tends to $0$
if $n$ tends to infinity. \\
Let $l_n$ be a sequence of positive integers such that $l_n/\sqrt{n}$ tends to infinity if $n$ tends to infinity.
The density of the set of weakly $l_n$-normal functions from $\V_n$ to $\F_p$ of degree at most $3$ in the set of all functions 
of degree at most $3$, tends to $0$ if $n$ tends to infinity.
\end{proposition}
{\it Proof.} For the proof we may identify $\V_n$ with $\F_p^n$.
The number of linear subspaces of $\F_p^n$ of dimension $k_n$ is
\[ \left[ \begin{array}{c} n \\ k_n \end{array} \right] = 
\frac{(p^n-1)(p^n-p)(p^n-p^2)\cdots (p^n-p^{k_n-1})}{(p^{k_n}-1)(p^{k_n}-p)(p^{k_n}-p^2)\cdots (p^{k_n}-p^{k_n-1})}, \]
hence the number of $k_n$-dimensional affine subspaces of $\F_p^n$ is
\[ \lambda_n = p^{n-k_n}\left[ \begin{array}{c} n \\ k_n \end{array} \right]. \]

Let $\mu_n$ be the number of functions from $\F_p^n$ to $\F_p$ which are affine on a fixed $k_n$-dimensional affine subspace
$A$ (which does not depend on the choice of $A$). To determine $\mu_n$, we choose $A=\F_2^{k_n} \times \{(0,\ldots,0)\}$.
Observing that the restriction of a $p$-ary function to $A$ is affine if and only if its ANF contains no monomial of degree at
least $2$ which only contains variables in $\{x_1,x_2,\ldots,x_{k_n}\}$. The number of such functions is 
$p^{p^n-p^{k_n}+k_n+1}$, hence the number $\omega_{k_n}$ of weakly-$k_n$-normal functions is at most 
\[ \lambda_n p^{p^n-p^{k_n}+k_n+1} = p^{n-k_n} \left[ \begin{array}{c} n \\ k_n \end{array} \right] p^{p^n-p^{k_n}+k_n+1}. \]
With 
\[ \left[ \begin{array}{c} n \\ k_n \end{array} \right] <  \frac{p^{nk_n-k_n^2+k_n}}{(p-1)^{k_n}} \le p^{nk_n-k_n^2+k_n-k_n\log_p2}, \]
we obtain that
\begin{eqnarray*} 
\omega_{k_n} & \le & \lambda_n p^{p^n-p^{k_n}+k_n+1} < p^{n-k_n}p^{nk_n-k_n^2+k_n-k_n{\rm log}_p2}
p^{p^n-p^{k_n}+k_n+1} \\
& = &  p^{p^n}p^{n(k_n+1)-k_n^2-k_n{\rm log}_p2+k_n+1-p^{k_n}} < p^{p^n}p^{n(k_n+1)-p^{k_n}}. 
\end{eqnarray*}
Since $\frac{p^{k_n}}{nk_n}$ tends to infinity when $n$ tends to $\infty$, the exponent $n(k_n+1)-p^{k_n}$ tends to 
$-\infty$. As a consequence, $\lim_{n\rightarrow\infty}\frac{\omega_{k_n}}{p^{p^n}} = 0$. 

Let $\nu_n$ be the number of functions from $\F_p^n$ to $\F_p$ of degree at most $3$ which are affine on $A=\F_2^{l_n} \times \{(0,\ldots,0)\}$.
Similarly as above we see that $\nu_n = p^{1+n+{n\choose 2}+{n\choose 3}-{l_n\choose 2}-{l_n\choose 3}}$, and the number of 
weakly-$l_n$-normal functions of degree at most $3$ is at most $p^{n(l_n+1)-l_n^2+1+n+{n\choose 2}+{n\choose 3}-{l_n\choose 2}-{l_n\choose 3}}$.
The density of this set in the set of $p$-ary functions of degree at most $3$ is therefore upper bounded by
\[ p^{n(l_n+1)-l_n^2-{l_n\choose 
2}-{l_n\choose 3}}, \]
which tends to $0$ if $n$ tends to infinity. \hfill$\Box$\\[.5em]
We remark that the proof of Proposition \ref{hnn} also shows that the existence of a not weakly $k$-normal function 
from $\V_n$ to $\F_p$ is guaranteed whenever
$\frac{p^{n(k+1)-k^2+k+1-p^k}}{(p-1)^k} < 1$.
For instance, there are not (weakly) normal functions for $p=3$ and $n=6$, and for $p=5$ and $n=4$.

For Boolean functions, in \cite{c04,c} relations between normality and Walsh coefficients have been
explored. We next generalize these results to $p$-ary functions.
Some identities for Boolean functions which play a role in the analysis can straightforwardly be generalized to odd $p$,
hence we omit the proof. Let
\begin{itemize}
\item[-] $V$ be a $k$-dimensional subspace of $\V_n$, and let $W$ be a complement of $V$ in $\V_n$, 
\item[-] $f_a$ be the restriction of $f:\V_n\rightarrow\F_p$ to $a+W$, i.e. $f_a(x) = f(a+x)$, $x\in W$,
\item[-] $D_bf(x) = f(x)-f(x+b)$ the derivative of $f$ in direction $b$.
\end{itemize}
Then
\begin{itemize}
\item[(a)] $\widehat{D_bf}(0) = \sum_{a\in V}\widehat{D_bf_a}(0)$ for any $b\in W$,
\item[(b)] $\sum_{u\in V}\widehat{f}(u+a)\overline{\widehat{f}(u+a)} = 
p^k\sum_{b\in V^\perp}\epsilon_p^{\langle a,b\rangle}\widehat{D_bf}(0)$ (Lemma V2 in \cite{cccf}),
\item[(c)] $\sum_{a\in V}\sum_{b\in W}\widehat{D_bf_a}(0) = \sum_{a\in V}\widehat{f_a}(0)\overline{\widehat{f_a}(0)}$.
\end{itemize}
The following lemma is the $p$-ary analog of Theorem V1 in \cite{cccf} (Equation (4) in \cite{c}).
\begin{lemma}
\label{ThV1}
Let $W$ be a $k$-dimensional subspace of $\V_n$ and let $V$ be a complement of $W$ in $\V_n$. Then
\[ \sum_{u\in W^\perp}\widehat{f}(u)\overline{\widehat{f}(u)} = p^{n-k}\sum_{a\in V}\widehat{f_a}(0)\overline{\widehat{f_a}(0)} %= p^{n-k}\sum_{a\in V}\sum_{x\in W}\epsilon_p^{f(a+x)}\sum_{y\in W}\epsilon_p^{-f(a+y)}
. \]
\end{lemma}
{\it Proof.} Applying (b), (a), (c) (in this order) we get $\sum_{u\in W^\perp}\widehat{f}(u)\overline{\widehat{f}(u)} = 
p^{n-k}\sum_{b\in W}\widehat{D_bf}(0) = p^{n-k}\sum_{b\in W}\sum_{a\in V}\widehat{D_bf_a}(0) = p^{n-k}\sum_{a\in V}\sum_{b\in W}\widehat{D_bf_a}(0)
= p^{n-k}\sum_{a\in V}\widehat{f_a}(0)\overline{\widehat{f_a}(0)}$. \hfill$\Box$\\[.5em]
The next lemma is a $p$-ary version of \cite[Corollary V3]{cccf}.
\begin{lemma}
\label{CV3}
With the above notations we have
\[ \sum_{a\in V}|\widehat{f_a}(0)|^2 \le \max_{u\in\V_n}|\widehat{f}(u)|^2. \]
Moreover, 
\[ \max_{v\in V_n}|\widehat{f_a}(v)| \le \max_{u\in\V_n}|\widehat{f}(u)|. \]
\end{lemma}
{\it Proof.}
By Lemma \ref{ThV1}, with $|W^\perp| = p^{n-k}$, we have 
$p^{n-k}\sum_{a\in V}|\widehat{f_a}(0)|^2 = \sum_{u\in W^\perp}|\widehat{f}(u)|^2 \le p^{n-k}\max_{u\in\V_n}|\widehat{f}(u)|^2$.
This in particular implies
\begin{equation} 
\label{lemax}
|\widehat{f_a}(0)| = |\sum_{x\in W}\epsilon_p^{f(x+a)}| \le \max_{u\in\V_n}|\widehat{f}(u)|
\end{equation}
for all $a \in V$. We may apply the same arguments to the function $\tilde{f}(x) = f(x) + \langle v, x\rangle$ for some $v\in\V_n$ 
(which has the same Walsh spectrum as $f$, 
hence $\max_{u\in\V_n}|\widehat{\tilde{f}}(u)| = \max_{u\in\V_n}|\widehat{f}(u)|$). Then $(\ref{lemax})$ converts to
\begin{equation}
\label{lemax1}
|\widehat{\tilde{f}_a}(0)| = |\sum_{x\in W}\epsilon_p^{f(x+a)+\langle v,x\rangle + \langle v,a\rangle}| \le \max_{u\in\V_n}|\widehat{f}(u)|
\end{equation}
for all $a \in V$, and the claim of the lemma follows. \hfill$\Box$. \\[.5em]
With Lemma \ref{CV3} we get the relation between normality and Walsh coefficient more general for functions from $\V_n$ to $\F_p$ for arbitrary primes $p$. 
\begin{corollary}
\label{normbound}
Let $f$ be a function from $\V_n$ to $\F_p$. If $f$ is (weakly) $k$-normal, then $p^k \le \max_{u\in\V_n}|\widehat{f}(u)|$. 
\end{corollary}
{\it Proof.}
Suppose that $f$ is weakly $k$-normal, i.e. $f(x) = \langle v, x\rangle +c$, for some $v\in\V_n$, $c\in\F_p$, and all $x\in a+W$ for some 
$k$-dimensional subspace $W$ of $\V_n$ and some $a$ in a complement $V$ of $W$. Then, using Lemma \ref{CV3} we have
\[ |\sum_{x\in W}\epsilon_p^{f(x+a)+\langle v,x\rangle + \langle v, a\rangle}| = p^k \le \max_{u\in\V_n}|\widehat{f}(u)|. \]
\hfill$\Box$\\[.5em]
For a bent function $f:\V_n\rightarrow \F_p$, Corollary \ref{normbound} implies that $f$ can be at most $\lfloor n/2 \rfloor$-normal.
Moreover, in \cite{cmp} 
it has been shown that a bent function in even dimension which is weakly regular but not regular cannot be normal. Hence, such a bent function can be
at most ($n/2-1$)-normal. However, whereas an arbitrary $p$-ary function is with high probability ``highly non-normal'' (see Proposition \ref{hnn}), 
many bent functions in odd characteristic are (weakly) $k$-normal with $k$ as large as the theory allows. This means, a $p$-ary bent function in even 
dimension $n$ seems most likely to be normal, unless it is weakly regular but not regular, in which case it would be  ($n/2-1$)-normal. 
A $p$-ary bent function in odd dimension $n$ 
appears likely to be $(n-1)/2$-normal. As in the case of Boolean bent functions, it may not be easy to find bent functions for which one 
can show a different behaviour. 
The following results on normality of $p$-ary bent functions support this point of view. Note that the large classes of completed Maiorana-McFarland 
and $PS^+$ bent functions (all of which members are regular bent functions in even dimension) are normal by their definition.
\begin{itemize}
\item[-] A quadratic bent function $Q:\V_n\rightarrow\F_p$, $p$ odd, is normal if $n$ is even and $Q$ is regular, ($n/2-1$)-normal if $n$ is even and $Q$ is weakly
regular but not regular, and $(n-1)/2$-normal if $n$ is odd, see \cite{cmp}.
\item[-] \cite[Proposition 5]{jzhl} A regular bent function of the form $f(x) = \Tr\left(\alpha x^{l(p^{n/2}-1)}\right) + \epsilon x^{(p^n-1)/2}$ is normal.
(For the bentness conditions see \cite[Theorem 1]{jzhl}.)
\item[-] \cite[Theorem 7]{cmp} The regular Coulter-Matthews bent functions are normal.
\item[-] The secondary construction of non-weakly regular bent functions $f:\V_n\rightarrow\F_p$ in \cite{agw,aw}, yields (weakly) normal bent 
functions when $n$ is even and (weakly) $(n-1)/2$-normal bent functions when $n$ is odd.
\item[-] \cite[Example 1]{cmp} $f:\F_{3^4}\rightarrow\F_3$, $f(x) = \T_4(\omega^{10}x^{22}+x^4)$, $\omega$ primitive element of $\F_{3^4}$, is normal.
\end{itemize}
The last example presented in \cite{hk1}, was one of the first known examples for a non-weakly regular bent function. As pointed out in \cite{cmp}, 
the function does not have a bent dual. One may expect that this in some sense not smooth bent function exhibits a more chaotic behaviour, which 
however does not apply with regard to normality. We here remark that differently to Boolean functions in dimension $4$ (see \cite{c04}), functions
from $\F_{3^4}$ to $\F_3$ which are not weakly normal do exist. Examples are the quadratic bent functions from $\F_{3^4}$ to $\F_3$ which are
weakly regular but not regular, and then by the result in \cite{cmp} not weakly normal.
%(MAYBY DIMENSION 4 IS TOO SMALL).
%

Candidates for non-weakly normal bent functions may be other sporadic examples of non-weakly regular bent functions:
\begin{enumerate}
\item $g_1:\F_{3^6}\rightarrow \F_3$ with $g_1(x) = {\rm Tr}_6(\xi^7x^{98})$, where
$\xi$ is a primitive element of $\F_{3^6}$, \cite{hk},
\item $g_2:\F_{3^6}\rightarrow \F_3$ with $g_2(x) = {\rm Tr}_6(\xi^7x^{14}+\xi^{35}x^{70})$, where
$\xi$ is a primitive element of $\F_{3^6}$, \cite{hk2}.
\end{enumerate}
Recently, the first construction of non-weakly regular bent functions for which the dual is not bent was presented, see \cite{cmp1}. 
This construction may also provide candidates for non-weakly normal bent functions: \\
Let $1,\alpha,\beta \in\F_{p^n}$ be linearly independent over $\F_p$, and let $f(x) = \Tr(x^2)$, $h_1(x) = \Tr(\alpha x^2)$, $h_2(x) = \Tr(\beta x^2)$.
Then the bent function $F:\F_{2^n}\times\F_p^2\rightarrow\F_p$
\[ F(x,y_1,y_2) = f(x) + (y_1+h_1(x))(y_2+h_2(x)) \]
is in general non-weakly regular.

\section{Testing normality}% for $f:\F_3^n\rightarrow\F_3$}
\label{test}

It is not easy to show (weak) normality for a given function, and it is even harder to disprove (weak) normality. There is no approach known, 
how to prove non-weak-normality by hand. In \cite{cddl,lg}, to show the non-weak-normality of some Boolean bent function in dimension 
$10,12,14$, a computer algorithm is used, see \cite{cddl}. In this section, based on the principles of the algorithm for Boolean 
functions in \cite{cddl}, we develop an algorithm for $p$-ary functions. 

Similarly as in \cite{cddl} for Boolean functions, the strategy is to combine cosets of a subspace $U$ of dimension $s$
on which $f$ is a fixed constant $c$ to an affine subspace of dimension $s+1$ on which $f$ is constant $c$.
Differently to the Boolean case, where the union of two cosets of a linear subspace $U$ is always an affine subspace,
the union of $p$ cosets of a subspace $U$ of $\F_p^n$ is in general not an affine subspace. Hence the algorithm in
\cite{cddl} does not transfer straightforward to $p$-ary functions.
To generate a complete list of the cosets of a subspace $U$ (without repetitions) we fix a complement $U^c$ of $U$.
A partition of $\F_p^n$ into cosets of $U$ we then get as $\{a+U\;:\;a\in U^c\}$.
\begin{lemma}
\label{Uprime}
Let $U$ be a linear subspace of $\V_n=\F_p^n$ of dimension $s<n$, let $U^c$ be a complement
of $U$ and let $a_1,a_2,\ldots,a_p$ be distinct elements of $U^c$. Then the union 
\[ \bigcup_{i=1}^q(a_i+U) \]
is an affine subspace $a_1+U^\prime$ of dimension $s+1$, if and only if $\{a_1,a_2,\ldots,a_p\}$
is an affine subspace $\{a_1+(a_2-a_1)t\,:\,0\le t\le p-1\}$ of $U^c$. Then
\[ a_1+U^\prime = a_1 + \langle a_2-a_1\rangle + U. \]
In particular, for $p=3$, $\bigcup_{i=1}^p(a_i+U)$ is an affine subspace if and only if $a_1+a_2+a_3 = 0$. 
\end{lemma}
\begin{proof}
First assume that $\{a_1,a_2,\ldots,a_p\}$ is an affine subspace which w.l.o.g.
we can write as $a_1 + \langle a_0\rangle$ with $a_0=a_2-a_1$. Then
\[
  \bigcup_{i=1}^p (a_i+U) = \bigcup_{t=0}^{p-1} (a_1+t(a_2-a_1)+U) = a_1 +
                            \bigcup_{i=0}^{p-1} (t(a_2-a_1)+U) = a_1+\langle a_0\rangle + U. \]
%where the last union is closed under addition, scalar multiplication of vectors and contains 
%$U, a_0$, and hence is equal to $U'=U\oplus\left<a_0\right>$.
Since $0\ne a_0 = a_2-a_1 \in U^c$ implies $a_0\not\in U$, the dimension of 
$U^\prime := \langle a_0\rangle + U$ is $s+1$. \\
Conversely, let the union $\bigcup_{i=1}^p(a_i+U) = a_1+U^\prime$ be an affine subspace for some
pairwise distinct $a_1,\ldots, a_p\in U^c$. Again we have $a_0=a_2-a_1\not\in U$, but $a_2-a_1\in U^\prime$.
Hence we can write $U^\prime$ as $\langle a_0\rangle + U$. For $1 < s\le p$ we can write the element $a_s-a_1$
of $U^\prime$ as $a_s-a_1 = u+ta_0$ for some $t\in\F_p$ and $u\in U$. Hence $\gamma = a_s-a_1-ta_0 = u \in U$.
Since $\gamma \in U^c$ we must have $\gamma = u = 0$, and hence $a_s = a_1 + ta_0$.
\end{proof}
\begin{lemma}
\label{aoU}
Let $f$ be a function from $\V_n$ to $\F_p$ and $A=a_1+U'$ be an affine subspace of dimension $s+1\le n$ of
$\V_n$. Then the restriction of $f$ to $A$ is affine but nonconstant if and only if $U^\prime = \langle a_0 \rangle + U$
such that $f$ is constant on each coset $(a_1+ta_0)+U$ of $U$, and affine (but nonconstant) on $a_1+\langle a_0\rangle$.
%\[
% A = \bigcup_{i=1}^q (a_i+U), \quad f(a_i+U)=\{i-1\},
%\]
%where $U'=U\oplus\left<a_1\right>$, $a_1,\dots,a_q\in U^c$.
\end{lemma}
\begin{remark}
The function $f$ is then constant on the cosets $a_1+ta_0 + U$ of $U$, $0\le t\le p-1$, with pairwise distinct
constants for pairwise distinct $0\le t_1,t_2\le p-1$.
For the special case that $p=3$, the condition in Lemma \ref{aoU} simplifies: 
The function $f$ is affine (but not constant) on $a+U$ if and only if $a+U$ is the union of three affine subspaces $a_1+U^\prime$, 
$a_2+U^\prime$, $a_3+U^\prime$ for a subspace $U^\prime$ of $\F_3^m$ of dimension $s-1$, such that $f_{|(a_1+U^\prime)}=c$,  
$f_{|(a_2+U^\prime)}=c+1$ and $f_{|(a_3+U^\prime)}=c+2$.
\end{remark}
{\it Proof of the Lemma.} Let $f$ be affine on $A$, i.e. there exists a linear function $L$ such that
$f(a_1+u') = L(u')+f(a_1)$ for $u'\in U^\prime$. 
Since we suppose that $f$ is not constant on $A$, the linear function
$L$ is not the zero-function on $U^\prime$, hence has an $s$-dimensional kernel $U$ in $U^\prime$. We can write
$U^\prime$ as $U^\prime = \langle a_0\rangle + U$ for some $a_0\in U^\prime\setminus U$, and observe that for
all $t\in\F_p$ and $u\in U$, 
\[ f(a_1+ta_0+u) = L(ta_0+u) + f(a_1) = tL(a_0)+L(u)+f(a_1) = tL(a_0)+f(a_1). \]
In particular, $f$ is affine on $a_1+\langle a_0\rangle$, and constant $tL(a_0)+f(a_1)$ on $a_1+ta_0+U$ for
every fixed $t$. \\
Conversely let $A=a_1+\langle a_0\rangle + U$, and suppose that $f$ is constant on $(a_1+ta_0)+U$ for every
fixed $0\le t\le p-1$, and affine on $a_1+\langle a_0\rangle$. Then for some linear function $L$ we have
\[ f(a_1+ta_0+u) = f(a_1+ta_0) = tL(a_0)+f(a_1), \]
hence affine on $a_1+U^\prime$. (Note that $U$ is in the kernel of $L$).
\hfill$\Box$

Lemma \ref{Uprime} and Lemma \ref{aoU} suggest the following procedure to construct an affine 
subspace of dimension $s+1$ on which $f:\F_p^n\rightarrow\F_p$ %$f:\F_3^n\rightarrow\F_3$
is constant, from such affine subspaces of dimension $s$. For a linear subspace $U$ of dimension $s$ fix a complement $U^c$ and find
$a_1,\dots,a_p\in U^c$ such that $f$ is constant with the same $c$ on all affine subspaces 
$a_1+U,\dots,a_p+U$. %Check if $f$ is also constant $c$ on $(2a_1+2a_2)+U$. If so
If $\{a_1,\dots a_p\}$ form a one-dimensional affine subspace, then take the union of those cosets. Note that this union then equals 
$a_1+U^\prime$ with $U^\prime = \langle U,a_2-a_1\rangle$. (In the following we use the term $1$-flat for a
one-dimensional affine subspace.)\\[.7em]
{\bf Basic Algorithm}:  \\[.5em]
{\bf Input:} a function $f:\F_p^n\rightarrow\F_p$, a starting dimension $s$ \\[.3em]
{\bf Output:} a list of all affine subspaces if dimension $m$ on which $f$ is affine \\[.3em]
{\bf For all} subspaces $U$ of $\F_p^n$ with $\dim(U) = s$ {\bf do} \\[.3em]
\hspace*{.5em} {\bf Fix} $\mathbf{\it U^c}$, a complement of $U$ in $\F_p^n$ \\[.3em]
\hspace*{.5em} {\bf Collect} all affine subspaces $a+U$, $a\in U^c$, with $f|_{a+U}=c$,
 for all $c\in\F_{p}$ \\[.3em]
\hspace*{.5em} {\bf Combine} tuples $(a_1+U,\dots,a_p+U)$ with 
$f|_{a_1+U}=\dots=f|_{a_p+U}=c$, for $c = 0,...,p-1$ resp., where the coset leaders form a $1$-flat
to get affine subspaces 
$a_1+U^\prime = a_1 + \langle U,a_2-a_1\rangle$ of dimension $s+1$ such that
$f|_{a_1+U^\prime}=c$, with $c=0,\dots,p-1$ resp. \\[.3em]
{\bf Repeat} these steps for the obtained subspaces $U^\prime$ up to dimension $m-1$ \\[.3em]
{\bf Combine} tuples $(a_1+U,\dots,a_p+U)$ where the $a_{i}$ form a $1$-flat and
with $f|_{a_i+U}=i$, $i = 0,\dots,p-1$ resp., or with
$f|_{a_1+U}=c, \dots, f|_{a_p+U}=c,$ with $c=0,\dots,p-1$
to get affine subspaces $a_1+U^\prime = a_1 + \langle U,a_2-a_1\rangle$ of dimension $m$ on which $f$ is affine \\[.3em]
{\bf Output} these affine subspaces of dimension $m$ \\[1.em]

In the case of $p=3$, we can employ some simplifications. We give this
algorithm in some more detail as well. \\[.7em]

{\bf Specific Algorithm for p=3:} (following Algorithm 1 in \cite{cddl}) \\[.5em]
{\bf Input:} a function $f:\F_3^n\rightarrow\F_3$, a starting dimension $s$ \\[.3em]
{\bf Output:} a list of all affine subspaces if dimension $m$ on which $f$ is affine \\[.3em]
{\bf For all} subspaces $U$ of $\F_3^n$ with $\dim(U) = s$ {\bf do} \\[.3em]
\hspace*{.5em} {\bf Fix} $\mathbf{\it U^c}$, a complement of $U$ in $\F_3^n$ \\[.3em]
\hspace*{.5em} {\bf Determine} all affine subspaces $a+U$, $a\in U^c$, with $f|_{a+U}=0$ and $f|_{a+U}=1$ and $f|_{a+U}=2$ resp. \\[.3em]
\hspace*{.5em} {\bf Combine} triples $(a_1+U,a_2+U,(2a_1+2a_2)+U)$ with $f|_{a_1+U}=f|_{a_2+U}=f|_{(2a_1+2a_2)+U}=c$, $c = 0,1,2$ resp.
to get affine subspaces $a_1+U^\prime = a_1 + \langle U,a_2-a_1\rangle$ of dimension $s+1$ such that
$f|_{a_1+U^\prime}=c$, $c=0,1,2$ resp. \\[.3em]
{\bf Repeat} these steps for the obtained subspaces $U^\prime$ up to dimension $m-1$ \\[.3em]
{\bf Combine} triples $(a_1+U,a_2+U,(2a_1+2a_2)+U)$ with $f|_{a_1+U}=f|_{a_2+U}=f|_{(2a_1+2a_2)+U}=c$, $c = 0,1,2$ resp., or with
$f|_{a_1+U}=c, f|_{a_2+U}=c+1, f|_{(2a_1+2a_2)+U}=c+2$
to get affine subspaces $a_1+U^\prime = a_1 + \langle U,a_2-a_1\rangle$ of dimension $m$ on which $f$ is affine \\[.3em]
{\bf Output} these affine subspaces of dimension $m$ \\[1.em]
We applied our algorithm to several known bent functions, and observed that many of them are in fact weakly $k$-normal with
$k$ as large as the theory allows. But we also found examples with a different behaviour. We collect some of the experimental 
results, which we find interesting in the following. For the first two examples we chose bent functions which have
maximal possible normality. The other functions we present below do not meet the upper bound on $k$-normality.
\begin{itemize}
\item[I] The weakly regular and non-regular Coulter-Matthews bent function\\ $\T_6(\xi^3 x^{{\lfloor(3^{7}+1)/2)\rfloor}})$ from $\F_{3^6}$ to $\F_3$,
where $\xi$ is a primitive element of $\F_{3^6}$, is $2$-normal.
\item[II] The regular bent function in dimension $4$,
 $\T_4(\xi^{138}{x^{24}}+\xi^{184}{x^{336}})$, 
 from $\F_{5^4}$ to $\F_5$,where $\xi$ is a primitive element of $\F_{5^4}$, is $2$-normal
 (Ex.7.1 in \cite{LiHellTaKho13}).
\item[III] The weakly regular Coulter-Matthews bent function in odd dimension $7$,\\ $\T_7(\xi^6 x^{\lfloor(3^{9}+1)/2)\rfloor})$, where $\xi$ is a primitive element
of $\F_{3^7}$, is $2$-normal but not (weakly) $3$-normal.
\item[IV] The weakly regular Coulter-Matthews bent function in odd dimension $9$,\\ $\T_9(\xi^5 x^{\lfloor(3^{11}+1)/2)\rfloor})$, where $\xi$ is a primitive element of $\F_{3^9}$, is $3$-normal but not (weakly) $4$-normal.
\item[V] The non-weakly regular bent function 
$g_1:\F_{3^6}\rightarrow \F_3$ with $g_1(x) = {\rm Tr}_6(\xi^7x^{98})$
where $\xi$ is a primitive element of $\F_{3^6}$, are not normal. 
\item[VI] The non-weakly regular bent function
$g_2:\F_{3^6}\rightarrow \F_3$ with $g_2(x) = {\rm Tr}_6(\xi^7x^{14}+\xi^{35}x^{70})$, 
where $\xi$ is a primitive element of $\F_{3^6}$, are not normal.
\item[VII] The non-weakly regular bent function $F:\F_{3^4}\times\F_p^2\rightarrow \F_3$ with
$F(x,y_1,y_2) = {\rm Tr}_4(x^2) + (y_1+{\rm Tr}_4(\xi^73x^2) )(y_2+{\rm Tr}_4(\xi^76x^2) )$,
where $\xi$ is a primitive element of $\F_{3^4}$, is not normal.
\end{itemize}

Examples III and IV are both bent functions in odd dimension, which are not (weakly) $(n-1)/2$-normal.
As our experimental results indicate, being solely $((n-1)/2-1)$-normal seems to be the typical behaviour
of Coulter-Matthews bent functions in odd dimension.
To the best of our knowledge, the last three examples are the first (non-binary) examples of bent functions in even dimension 
(not in the class of weakly regular but not regular bent functions) which are shown to be not (weakly) normal. 
All functions in V,VI,VII are non-weakly regular bent functions for which the dual is not bent, see \cite{cmp0,cmp1}.

%One may use these functions to generate non-(weakly) normal bent functions in even dimension $\ge 6$ with a ternary version of Lemma 25 in \cite{cddl}. 

Once a bent function in dimension $n$ is known which is not (weakly) $k$-normal for some $k$, we can construct bent functions
in any dimension $N = n+2s$, $s\ge 1$, that is not (weakly) ($k+s$)-normal, applying the following lemma which is a generalization
of Lemma 25 in \cite{cddl} for Boolean functions in dimension $n$ and $k=n/2$. In particular we can construct not weakly normal
(not weakly $(N-1)/2$, $N/2-1$-normal) bent functions in dimension $N$ from such bent functions in dimension $n$. 
\begin{lemma}
\label{25}
For a $p$-ary function $f:\F_p^n\rightarrow\F_p$ the following properties are equivalent.
\begin{itemize}
\item[(1)] $f$ is (weakly) $k$-normal,
\item[(2)] $g:\F_p^n\times\F_p^2\rightarrow\F_p$ given by $g(x,y,z) = f(x) + yz$ is (weakly) $(k+1)$-normal.
\end{itemize}
In particular, $f$ is (weakly) normal if and only if $g$ is weakly normal ($n$ even).
\end{lemma}
{\it Proof.}
First suppose that $f$ is (weakly) $k$-normal, and $E$ is a $k$-dimensional affine subspace restricted to
which $f$ is constant (affine). Then $g$ is constant (affine) on the $(k+1)$-dimensional affine subspace 
$E^\prime = \{(x,y,0) \,:\,x\in E, y\in\F_p\}$ of $\F_p^n\times\F_p^2$.

Conversely suppose that $g$ is weakly $(k+1)$-normal, and let $E^\prime = w + U^\prime$, $w=(w_1,w_2,w_3)$, be a
$(k+1)$-dimensional affine subspace of $\F_p^n\times\F_p^2$ restricted to which $g$ is 
constant or affine. Then for $(x,y,z) \in E^\prime$ we have
\begin{equation}
\label{onE}
g(x,y,z) = \langle\gamma,x\rangle + \alpha y + \beta z + c
\end{equation}
for some $\gamma \in\F_p^n$, $\alpha,\beta,c\in\F_p$. For $a,b\in\F_p$ define
\begin{equation}
\label{Eab}
E_{ab} = \{x\in\F_p^n\,:\,(x,a,b) \in E^\prime\}.
\end{equation}
If $\bar{x}\in E_{ab}$, then $E_{ab} = \bar{x} + U$, where $U$ is the subspace of $\F_p^n$ given by
$U = \{x\in\F_p^n\,:\, (x,0,0) \in U^\prime\}$ (straightforward). Observe that restricted to $E_{ab}$,
the function 
\begin{equation}
\label{fonE}
f(x)-\langle\gamma,x\rangle = \alpha a + \beta b + c - ab
\end{equation}
is constant. If $U$ has dimension $k$ we are done. Suppose that $\dim(U) < k$. Since $E^\prime$
is the union $\bigcup_{a,b}\{(x,a,b)\,:\,x\in E_{ab}\}$ (some $E_{ab}$ may be the same, some the empty set),
we have $p^{k+1} = |E^\prime| \le \sum_{a,b}|E_{a,b}|$. As we assume that $\dim(U) < k$, this implies that
$\dim(U) = k-1$, i.e. $|E_{ab}| = p^{k-1}$ for all $(a,b)\in\F_p^2$ and all $E_{ab}$ are distinct. We then 
define $E$ as the disjoint union
\[ E = \bigcup_{a\in\F_p}E_{a\alpha} = \bar{x} + \bar{U} \]
for an element $\bar{x}\in E$, where $\bar{U} = \{(x\in\F_p^n\,:\,(x,a,0)\in E^\prime\;\mbox{for some}\;a \in \F_p\}$,
and observe that $f(x)-\langle\gamma,x\rangle = \beta\alpha + c$ is constant on this $k$-dimensional affine
subspace. \hfill$\Box$\\[.5em]

\section{Perspectives}

In this article we contribute to the analysis of $k$-normality for $p$-ary bent functions.
Depending on the regularity of a bent function $f$ from $\V_n$ to $\F_p$ and the parity of $n$,
many bent functions seem to be (weakly) normal, $(n/2-1)$-normal or $(n-1)/2$-normal, which is
drastically different from the average behaviour of a $p$-ary function.
It seems not easy to find bent functions for which one can show a different behaviour. This
resembles the situation for Boolean bent functions. We develop an algorithm for testing normality
for $p$-ary functions. Applying this algorithm we verify that some ternary non-weakly regular bent 
functions in even dimension $n$ are not weakly normal. For odd dimension $n$ we found examples
in the class of Coulter-Matthews bent functions which are not weakly $(n-1)/2$-normal.
With Lemma \ref{25} we then can construct from such functions in dimension $n$, bent functions 
with the same property in any dimension $n+2s$, $s\ge 1$.

There are many interesting open questions on normality for $p$-ary bent functions. We close with 
a collection of some of them, which can now be attacked using our presented algorithm.
\begin{itemize}
\item[-] Find regular $p$-ary bent functions in even dimension which are not normal. 
\item[-] Find weakly regular but not regular $p$-ary bent functions in even dimension which are not 
$(n/2-1)$-normal.
%\item[-] Find $p$-ary bent functions in odd dimension which are not $((n-1)/2)$-normal.
\item[-] Show that the weakly regular but not regular Coulter-Matthews bent functions 
in even dimension are $(n/2-1)$-normal or find counter-examples. 
%, and the Coulter-Matthews bent functions in odd dimension are $((n-1)/2)$-normal,
%or find counter-examples. 
\end{itemize}
The question on the average behaviour of Boolean and $p$-ary bent functions with respect to normality
seems not easy to be answered. Are (most) bent functions affine on affine subspaces of large dimension,
or do they behave like arbitrary Boolean and $p$-ary functions, normal, $(n/2-1)$-normal,
$((n-1)/2)$-normal bent functions are only easier to find?


\begin{thebibliography}{99}



\bibitem{cccf} A. Canteaut, C. Carlet, P. Charpin, C. Fontaine. On the cryptographic properties of $R(1,m)$. 
{\em IEEE Trans. Inform. Theory}, 47: 1494--1513, 2001.

\bibitem{cddl} A. Canteaut, M. Daum, H. Dobbertin, G. Leander. Finding nonnormal bent functions. {\em Discrete Appl. Math.}, 
154: 202--218, 2006.

\bibitem{c04} C. Carlet. On the degree, nonlinearity, algebraic thickness, and nonnormality of Boolean functions, with
developments on symmetric functions. {\em IEEE Trans. Inform. Theory}, 50: 2178--2185, 2004.

\bibitem{agw} A. \c Ce\c smelio\u glu, G. McGuire, and W. Meidl. A construction of weakly and
non-weakly regular bent functions. {\em J. Combin. Theory Ser. A}, 119: 420--429, 2012.

\bibitem{aw} A. \c Ce\c smelio\u glu and W. Meidl. A construction of bent functions from plateaued functions.  
{\em Des. Codes Cryptogr.}, 66: 231--242, 2013.

\bibitem{cmp0} A. \c Ce\c smelio\u glu, W. Meidl, A. Pott. On the dual of (non)-weakly regular bent functions and self-dual bent functions. 
{\em Adv. Math. Commun.}, 7: 425--440, 2013.

\bibitem{cmp} A. \c Ce\c smelio\u glu, W. Meidl, A. Pott. Generalized Maiorana-McFarland class and normality of $p$-ary bent functions. 
{\em Finite Fields Appl.}, 24: 105--117, 2013.

\bibitem{cmp1} A. \c Ce\c smelio\u glu, W. Meidl, A. Pott. There are infinitely many bent functions for which the dual is not bent. 
{\em IEEE Trans. Inform. Theory}, 62: 5204--5208, 2016.

\bibitem{c} P. Charpin. Normal Boolean functions. {\em J. Complexity}, 20: 245--265, 2004.

\bibitem{hk} T. Helleseth, A. Kholosha. Monomial and quadratic bent functions over the finite fields of odd characteristic.
{\em IEEE Trans. Inform. Theory}, 52: 2018--2032, 2006.

\bibitem{hk1} T. Helleseth, A. Kholosha. New bionomial bent functions over the finite fields of odd characteristic.
{\em IEEE Trans. Inform. Theory}, 56: 4646--4652, 2010.

\bibitem{hk2} T. Helleseth, A. Kholosha. Crosscorrelation of m-sequences, exponential sums, bent functions and Jacobsthal sums. 
{\em Cryptogr. Commun.}, 3: 281--291, 2011.

\bibitem{jzhl} W. Jia, X. Zeng, T. Helleseth, C. Li. A class of binomial bent functions over the finite field of odd characteristic.
{\em IEEE Trans. Inform. Theory}, 58: 6054--6063, 2012.

\bibitem{lg} G. Leander, G. McGuire. Construction of bent functions from near-bent functions. {\em J. Combin. Theory Ser. A},
116: 960--970, 2009.

\bibitem{LiHellTaKho13} N. Li, T. Helleseth, X. Tang, A. Kholosha.
Several new classes of bent functions from Dillon exponents.
{\em IEEE Trans. Inform. Theory}, 59: 1818--1831, 2013. 

\bibitem{mgpi} M.J. Mihaljevic, S. Gangopadhyan, G. Paul, H. Imai. Generic cryptographic weakness of $k$-normal Boolean functions 
in certain stream ciphers and cryptanalysis of Grain-128. {\em Periodica Mathematica Hungarica}, 65: 205--227, 2012. 



\end{thebibliography}
\end{document}